\documentclass[]{article}
\usepackage{graphicx}
\usepackage{graphics}
\usepackage{amssymb,amsmath,amsthm,epsfig}
\newtheorem{theorem}{Theorem}
\newtheorem{lemma}{Lemma}
\newtheorem{corollary}{Corollary}
\newtheorem{definition}{Definition}
\newtheorem{proposition}{Proposition}

\newtheorem{remark}{Remark}
\newcommand{\R}{\mathbb R}

\newcounter{rea}
\newcounter{rek}
\begin{document}


%
\begin{center}
{\large {\bf Approximations in Sobolev Spaces by
Prolate Spheroidal Wave Functions.}}\\
\vskip 1cm Aline Bonami$^a$ and Abderrazek Karoui$^b$ {\footnote{Corresponding author,\\
This work was supported in part by the  French-Tunisian  CMCU 10G 1503 project and the
Tunisian DGRST  research grants 05UR 15-02 and UR 13ES47. Part of this work was done while the first author was visiting the Faculty of Sci. of Bizerte, Univeristy of Carthage.
The two authors have also benefited from the program " Research in pairs " of the CIRM, Luminy, France.}}
\end{center}
\vskip 0.5cm {\small
$^a$ F\'ed\'eration Denis-Poisson, MAPMO-UMR 7349,  Department of Mathematics, University of Orl\'eans, 45067 Orl\'eans cedex 2, France.\\
\noindent $^b$ University of Carthage,
Department of Mathematics, Faculty of Sciences of Bizerte, Tunisia.}\\
Emails: aline.bonami@univ-orleans.fr ( A. Bonami), abderrazek.karoui@fsb.rnu.tn (A. Karoui)\\

\noindent{\bf Abstract}--- Recently, there is a growing interest in
the spectral approximation by the   Prolate Spheroidal Wave
Functions (PSWFs) $\psi_{n, c},\, c>0.$ This is due to the promising new contributions of these
functions in various classical as well as emerging applications from Signal Processing,
Geophysics, Numerical
Analysis, etc. The PSWFs form a basis with remarkable properties not only
for the space of band-limited functions with bandwidth $c,$ but also for the Sobolev
space $H^s([-1,1])$. The quality of the spectral approximation and the choice of the parameter $c$
when approximating a function in  $H^s([-1,1])$ by its truncated PSWFs series expansion, are
the main issues. By considering a function  $f\in H^s([-1,1])$ as
the restriction to $[-1,1]$ of an almost time-limited  and
band-limited function, we try to give satisfactory answers to these two issues.
Also, we illustrate the different results of this work by some
numerical examples.\\


\section{Introduction}

Let $f$ be a function that belongs to some Sobolev space $H^s(I),
s>0, \, I=[-1,1]$. The main issue of this work concerns the speed of convergence in $L^2(I)$ of its expansion in some  PSWF basis.
\smallskip

Let us recall that, for a given value $c>0$, called the bandwidth, PSWFs $(\psi_{n,c})_{n\geq 0}$ constitute an orthonormal basis of $L^2([-1, +1])$ of eigenfunctions of the two compact integral operators
 $\mathcal F_c$ and $\mathcal Q_c= \frac c{2\pi}\mathcal F_c^*\mathcal F_c$, defined  on $L^2(I)$ by
\begin{equation}\label{eq1.1}
 \mathcal F_c(f)(x)= \int_{-1}^1
e^{i\, c\, x\, y} f(y)\, dy,\quad \mathcal Q_c(f)(x)=\int_{-1}^1\frac{\sin c(x-y)}{\pi (x-y)}\, f(y)\,
dy. \end{equation}
PSWFs are also eigenfunctions of the
 Sturm-Liouville operator $\mathcal L_c$, defined by
\begin{equation}\label{eq1.0}
\mathcal L_c(\psi)=-\frac{d}{d\, x}\left[(1-x^2)\frac{d\psi}{d\,x}\right]+c^2 x^2\psi,
\end{equation}
We call $\chi_n(c)$ the eigenvalues of  $\mathcal L_c$, and $\lambda_n(c)$ the eigenvalues of $\mathcal Q_c$. The first ones are  arranged in the increasing order, the second ones in the  decreasing order $1>
\lambda_0(c)> \lambda_1(c)>\cdots>\lambda_n(c)>\cdots.$ We finally call $\mu_n(c)$ the eigenvalues of $\mathcal F_c$. They are given by
$$\mu_{n}(c)=i^n\sqrt{\frac {2\pi}c \lambda_n(c)}.$$
By Plancherel identity, PSWFs are normalized so that
\begin{equation}\label{eqq1.4}
\int_{-1}^1 |\psi_{n,c}(x)|^2\, dx = 1,\quad \int_{\mathbb R}
|\psi_{n,c}(x)|^2\, dx =\frac{1}{\lambda_n(c)},\quad n\geq 0.
\end{equation}
We  adopt the  sign normalization of the PSWFs, given by
\begin{equation}\label{eeqq1.4}
\psi_{n,c}(0) > 0\mbox{\ \  for even\ \ } n,\quad \, \psi'_{n,c}(0) > 0,\mbox{\ \ for odd \ \ }  n.
\end{equation}
A breakthrough in the theory and the computation of the PSWFs goes back to the 1960's and is due to D. Slepian
and his co-authors H. Landau and H. Pollack. For the classical and more recent developments in the area of the PSWFs, the reader
is referred to the recent books on the subjects \cite{Hogan, Osipov3}.
This paper is a companion paper of \cite{Bonami-Karoui2} and we refer to it for further notations and references.
\smallskip

   This question of the quality of approximation has attracted a growing interest while, at the same time, were built
PSWFs based numerical schemes for solving various problems from
numerical analysis, see \cite{Boyd1, Boyd2, Chen, Lin,  Moore, Rokhlin2, Wang}.
In particular, in \cite{Boyd1}, the author has shown that a PSWF
approximation based method outperforms in terms of spatial
resolution and stability of time-step, the classical approximation
methods based on Legendre or Tchebyshev polynomials. The authors
of \cite{Chen} were among the first to compare the quality of
approximation by the PSWFs for different values of $c$. In particular, they have given an estimate of the decay of
the PSWFs expansion coefficients of a function $f\in H^s(I)$, see also
 \cite{Boyd1}. Recently, in \cite{Wang}, the author studied the speed of convergence of the expansion of such a function in a basis of PSWFs.  We should mention that the methods
used in the previous three references are heavily based on the use
of the properties of the PSWFs as eigenfunctions of the
differential operator $\mathcal L_c,$ given by (\ref{eq1.0}). They pose the problem of
 the best choice of the value
of the band-width $c>0,$ for
approximating well a given $f\in H^s(I)$, but their answer is mainly experimental.  It has
been numerically checked in \cite{Boyd1, Wang} that the smaller
the value of $s,$ the larger the  value of $c$ should be.

\medskip

Our study tries to give a satisfactory answer to this important problem
of the choice of the parameter $c.$
More precisely, we show that
if  $f\in
H^s(I)$, for some positive real number $s>0,$ then for any
integer $N\geq 1,$ we have
\begin{equation}
\| f-S_N f\|_{L^2(I)}\leq K(1+c^2)^{-s/2} \|
f\|_{H^s(I)}+ K\sqrt{ \lambda_N(c)} \|f\|_{L^2(I)}.
\end{equation}
Here, ${\displaystyle S_N f=\sum_{k=0}^N <f,\psi_{n,c}> \psi_{n,c}}$ and   $K$ is a constant depending only on $s.$ With this expression, one sees clearly how to distribute a fixed error between that part which is due to the smoothness of the function and that part which is due to the speed of convergence for the PSWFs.
We also study an  $L^2(I)-$convergence rate of the projection $S_N f$ to $f.$ This is done by using the
 decay of the eigenvalues $(\lambda_n(c))_n$ as well as estimates of Legendre expansion coefficients of PSWFs and the decay of Fourier coefficients of PSWFs.
We prove an exponential decay rate, given by
\begin{equation}
|\langle e^{ik\pi x}, \psi_{n,c}(x)\rangle|\leq M' e^{-an},\quad |k|\leq n/M,\quad n\geq \max\left( c M, 3\right).
\end{equation}
Here, $c\geq 1,$ $M\geq \sqrt{2}$ and $M', a >0$ are two positive constants. Under these hypotheses and notations, our rate of convergence
of $S_N f$ to $f\in H^s(I)$ is given by
\begin{equation}
\|f- S_N (f)\|_{L^2(I)}\leq M'' (1+ N^2)^{-s/2}  \| f\|_{H^s}+ M' e^{-aN}  \| f\|_{L^2}.
\end{equation}

\smallskip

This work is organized as follows. In Section 2,  we first  give  bounds for the moments of the PSWFs, which we use to improve estimates of the
decay of  the Legendre expansion  coefficients of the PSWFs.
  In Section 3, we first consider the
quality of approximation by the PSWFs in the set  of almost time
and band-limited functions. Then, we combine these results with
those of  Section 2 and give a first  $L^2(I)-$error bound of
approximating a function $f\in H^s(I)$ by the  $N-$th partial
sum of its PSWFs series expansion.
Then, we study a more elaborated error analysis of the spectral approximation
by the PSWFs in the periodic Sobolev space.
This is afterwards  extended to the usual Sobolev space $H^s(I).$
 These new estimates provide us with a way for
the choice of the appropriate bandwidth $c>0$ to be used by a
PSWFs based method for the approximation in a given Sobolev space
$H^s(I).$ In Section 4,  we provide the reader
with  some numerical examples that illustrate the different
results	 of this work.

 We will frequently  skip the parameter  $c$ in $\chi_n(c)$ and $\psi_{n,c}$, when there is no doubt on the value of the bandwidth.
 We then note  $q= c^2/\chi_n$ and skip both parameters $n$ and $c$ when their values are obvious from the context.

\section{Decay estimates of the Legendre expansion coefficients}

In this paragraph, we give bounds for the Legendre expansion coefficients of PSWFs, which will be used later and are of general interest.
Legendre expansion coefficients of PSWFs have been the object of many studies, in particular in relation with numerical methods for their evaluation.
In particular, the classical method known as Flammer's method, \cite{Flammer} that uses the differential operator $\mathcal L_c$, is extensively used
to compute the PSWFs and their eigenvalues.

The  Legendre expansion of the PSWFs is given by
\begin{equation}\label{eqqq3.2}
 \psi_{n}(x)={\sum_{k\geq 0}} \beta_k^n \overline{P_k}(x).
 \end{equation}
Recall that  $\psi_n$ has the same parity as $n.$ Hence, the previous Legendre expansion coefficients of the
$\psi_n$ satisfy $\beta_k^n=0$ if $n$ and $k$ have different parities.

 It is well known that the  different
expansion coefficients $(\beta_k^n)_k$ as well as the
corresponding eigenvalues $\chi_n$ are obtained by solving the
following eigensystem
\begin{eqnarray}
&&\frac{(k+1)(k+2)}{(2k+3)\sqrt{(2k+5)(2k+1)}} c^2 \beta_{k+2}^n
 + \big( k(k+1)+ \frac{2k(k+1)-1}{(2k+3)(2k-1)} c^2\big)
\beta_k^n \label{eigensystem}\\
&&\hspace*{3cm} + \frac{k(k-1)}{(2k-1)\sqrt{(2k+1)(2k-3)}} c^2
\beta_{k-2}^n= \chi_n(c) \beta_k^n, \quad k\geq 0.\nonumber
\end{eqnarray}
The eigenvalues $\chi_n$ satisfy the following well known bounds
\begin{equation}
\label{bounds0-chi}
n (n+1)     \leq \chi_n \leq n(n+1)+c^2.
\end{equation}
In the case where $q=c^2/\chi_n\leq 1,$ we have the following less classical bounds
\begin{equation}
\label{bounds2-chi}
n (n+1)+ (3-2\sqrt 2)c^2     \leq  \chi_n\leq \left(\frac{\pi}{2} (n+1)\right)^2.
\end{equation}
The above lower bound has been given in \cite{Bonami-Karoui1} and the upper bound is given in \cite{Osipov}. The previous inequalities
will be needed in different parts of this work.

The decay of the coefficients $\beta_{k}^n$ has been the object of many studies. We give here upper bounds of the  $|\beta_k^n|$ with $k$ varying  from $0$ to  a certain order depending on $n.$ Such bounds have been proposed in \cite{Chen}. There was a gap in their proof, which has been filled up in the book \cite{Hogan} with some loss in the results. We improve these estimates that have been proposed by \cite{Chen}, both for the range of $k$ and the constant. In view of these results we also prove auxiliary properties, which may be of independent interest. We first provide bounds for the successive derivatives of $\psi_n$  at $0.$
For $n=0$,  we have proved in \cite{Bonami-Karoui1} that, for $q<2$,
\begin{equation}
\label{bound1psi}
 |\psi_n(0)|^2+ \chi_n^{-1}|\psi'_n(0)|^2\leq 1.
\end{equation}
Recall that either $\psi_n(0)$ or $\psi'_n(0)$ is zero, depending on the parity. The same is valid for higher derivatives at $0$.

\begin{proposition} Assume that $ c^2 <\chi_n.$  Then for any  integer $ k\geq 0$ satisfying $k(k+1)\leq \chi_n,$ we have
\begin{equation}\label{ineqderivatives}
\left|\psi^{(k)}_{n}(0)\right|\leq (\sqrt{\chi_n})^k ( |\psi_n(0)|^2+ \chi_n^{-1}|\psi'_n(0)|^2)^{1/2}.
\end{equation}
\end{proposition}
\begin{proof} Since $\psi_{n, c}$ has same parity as $n$, then it is sufficient to consider even derivatives or odd derivatives depending on the parity of $n$. We assume  that  $n$ is even and consider $k=2l$. The proof is identical for odd values.
 By an
iterative use  of the identity
$$(1-x^2)\psi_n''(x)=2x \psi_n'(x)+(c^2x^2-\chi_n)\psi_n(x),$$
 one can easily check that the
${\displaystyle \psi^{(k)}_{n,c}(0)=\psi^{(k)}(0)}$ are given by
the  recurrence relation
\begin{equation}\label{recurrence1}
\psi^{(k+2)}(0)=(k(k+1)-\chi_n) \psi^{(k)}(0)+k(k-1) c^2
\psi^{(k-2)}(0),\quad k\geq 0.
\end{equation}
  Let us show by induction that for a fixed $n,$ ${\displaystyle
\psi^{(2l)}_{n,c}(0)}$ has alternating signs, that is
$\psi^{(k)}_{n,c}(0)\psi^{(k-2)}_{n,c}(0)<0.$ Indeed,
with $\psi(0) >0,$ $\psi^{(2)}(0)=-\chi_n \psi(0)$ so that
the induction hypothesis is fulfilled for $k=2$.
 Multiplying both sides of (\ref{recurrence1}) by $\psi^{(k)}(0),$ using the assumption that
$k(k+1)\leq \chi_n$ as well as the induction hypothesis, one
concludes that the induction assumption holds for the order $k.$
Consequently, we have,
\begin{equation}\label{recurrence2}
\left|\psi^{(k+2)}(0)\right|=(\chi_n-k(k+1))
\left|\psi^{(k)}(0)\right|+k(k-1) c^2
\left|\psi^{(k-2)}(0)\right|,\quad k\geq 0.
\end{equation}
This may be rewritten as
\begin{equation}\label{recurrence3}
m_{k+2}=\left(1-\frac{k(k+1)}{\chi_n}\right)m_k+k(k-1)\frac{q}{\chi_n}
m_{k-2}.
\end{equation}
The fact that all the $m_{2l}$ are bounded by $m_0=\psi(0)= m_2$ follows at once by induction.
For $n$ odd the proof follows the same lines. \end{proof}

As a consequence of the previous proposition, we have the following corollary concerning the sign and the bounds
of the different moments of the $\psi_n.$

\begin{corollary} Let $c>0,$ be a positive real number. We assume that $q= c^2/\chi_n<1$. Then, for $j(j+1)\leq \chi_n$, all moments $\int_{-1}^1 y^j \psi_{n}(y)\, dy$ are non negative and
\begin{equation}
\label{momentspsi}
0\leq\int_{-1}^1 y^{j} \psi_{n}(y)\, dy \leq \left(\frac{1}{q}\right)^{j} |\mu_n(c)|.
\end{equation}
\end{corollary}

\begin{proof} It is sufficient to consider moments of odd order when $n$ is odd and   of even order when $n$ is even.
By taking the $j-$th derivative at zero on both sides of ${\displaystyle \int_{-1}^1 e^{icxy} \psi_{n}(y)\, dy =\mu_n(c) \psi_{n}(x),}$
one gets
\begin{equation}
\label{moments2psi}
\int_{-1}^1 y^j \psi_{n}(y)\, dy= (-i)^j c^{-j} \mu_n(c)\psi_{n}^{(j)}(0).
\end{equation}
Since $\psi_{n}^{(j)}(0)$ and  $\psi_{n}^{(j+2)}(0)$ have opposite signs, then the previous equation implies
that  moments  have the same sign for any positive integer $j$ with $j(j+1)\leq \chi_n.$ The  inequality
 (\ref{momentspsi}) follows from the previous proposition.
\end{proof}

We now study the positivity of  $\beta_k^n$ for small values of $k$. Remark that for $c=0$ all these coefficients vanish when $k<n$. The positivity of $\beta_0^n$ (when $n$ is even)  and $\beta_1^n$ (when $n$ is odd) follows from the fact that
\begin{equation}
\label{beta0}
 \beta_0^n=\frac 1{\sqrt 2} \int_{-1}^{+1}\psi_n(y)dy= \frac 1{\sqrt 2}|\mu_n(c)|\psi_n(0),\qquad  \beta_1^n= \sqrt{\frac 32} \int_{-1}^{+1}y\psi_n(y)dy= |\mu_n(c)|\frac{\psi_n'(0)}{c}.
\end{equation}

 \begin{lemma}
Let $c>0,$ be a fixed positive real number. Then, for all positive  integers $k, n$ such that $k(k-1)+1. 13\, c^2\leq \chi_n(c)$, we have $\beta_k^n\geq 0$.
\end{lemma}

\begin{proof}
We recall that the $\beta_k^n$ are given by the  eigensystem (\ref{eigensystem}). Let us first consider $k=2$ (when $n$ is even) and $k=3$ (when $n$ is odd) and compute
$$ \beta_2^{n} =\frac{3\sqrt{5}}{2 c^2}\left(\chi_n -\frac{c^2}{3}\right)\beta_0^{n},\quad
 \beta_3^{n} =\frac{5\sqrt{21}}{6 c^2}\left(\chi_n-2 -\frac{3 c^2}{5}\right)\beta_1^{n}.$$
They are positive, assuming that $2$ satisfies the condition (resp. 3 satisfies the condition). For $k\geq 2$, taking upper bounds for the fractions as in  \cite{Chen}, Equation (\ref{eigensystem}) implies that
\begin{equation}\label{ineq}
\frac{2c^2}{3\sqrt{5}} ( \beta_{j+2}^n +\beta_{j-2}^n)\geq
 ( \chi_n(c)-j(j+1)-\frac{11c^2}{21}) \beta_j^n.
\end{equation}
The constant $1.13$ has been chosen so that $1.13>\frac{4}{3\sqrt{5}}+\frac{11}{21}$. Let us prove  by induction that the sequence $\beta_{2j}^n$ (resp. $\beta_{2j+1}^n)$ is non decreasing for $2j\leq k$ (resp. $2j+1\leq k$) when the assumption is satisfied by $k$. Without loss of generality we assume now that $n$ is even. We have $\beta_2^n\geq \beta_0^n$. Next, by the induction assumption on $j$, we know that $\beta_{j-2}^n\leq \beta_j^n$. We get a contradiction with inequality \eqref{ineq} if we assume  that $\beta_{j+2}^n< \beta_j^n$. So $\beta_{j+2}^n\geq \beta_j^n$.
 This implies the positivity.
\end{proof}

The following proposition provides us with  decay rate of the expansion coefficients $\beta_k^n$ that we had in view.

\begin{proposition}
Let $c>0,$ be a fixed positive real number. Then, for all positive integers  $n,k$ such that $k(k-1)+1. 13\,  c^2\leq \chi_n(c)$, we have
\begin{equation}
\label{Decay2beta}
|\beta_0^n|\leq \frac{1}{\sqrt{2}}|\mu_n(c)|\quad \mbox{ and }\quad |\beta_k^n|\leq \sqrt{\frac{5}{4\pi}}\left(\frac{2}{\sqrt{q}}\right)^{k} |\mu_n(c)|.
\end{equation}
\end{proposition}

\begin{proof} The first inequality follows from \eqref{bound1psi} and \eqref{beta0}.   To prove the second inequality, we first note that the moments $a_{jk}$ of the normalized Legendre polynomials $\overline P_k$ are non-negative (see \cite{Andrews}). They vanish except for $k\leq j$, with $k, j$ of the same parity. Moreover, for $j=k$, we have
\begin{equation}
a_{kk}=\int_{-1}^{1}x^{k}\overline{P_k}(x)dx=
\frac{\sqrt{\pi}\sqrt{k+1/2} \,k!}{2^{k}\Gamma(k+\frac{3}{2})}.
\end{equation}
Since ${\displaystyle x^j=\sum_{k=0}^j a_{jk} \overline{P_k}(x),}$  the moments of the $\psi_n$ are related to the
PSWFs Legendre expansion coefficients by
$$\int_{-1}^1 x^j \psi_n(x)\, dx = \sum_{k=0}^j a_{jk}  \beta_k^n.$$
Since by the previous lemma, we have $\beta_k^n \geq 0,$ for any $0\leq  k\leq j$ and since the $a_{jk}$ are non negative,  the previous equality
implies that
\begin{equation}\label{decay3beta}
\beta_j^n \leq \frac{1}{a_{jj}} \int_{-1}^1 x^j \psi_n(x)\, dx \leq  \frac{1}{a_{jj}} \left(\frac{1}{q}\right)^{j/2}|\mu_n(c)|.
\end{equation}
The last inequality follows from the previous corollary. On the other hand, we have
$$a_{jj}= \frac{\sqrt{\pi} \sqrt{j+1/2} j! }{2^j \Gamma(j+3/2)}= \frac{\sqrt{\pi}  j! }{2^j \sqrt{j+1/2}\Gamma(j+1/2)}.$$
Moreover, it is well known that ${\displaystyle j^{1-s}\leq \frac{\Gamma(j+1)}{\Gamma(j+s)}\leq (j+1)^{1-s}}.$ Hence, we have
\begin{equation}\label{decay4beta}
\frac{1}{a_{jj}}\leq \frac{2^j}{\sqrt{\pi}}\sqrt{1+\frac{1}{2j}}\leq 2^j\sqrt{\frac{5}{4\pi}},\quad\forall\, j\geq 1.
\end{equation}
By combining (\ref{decay3beta}) and (\ref{decay4beta}), one gets the second inequality of (\ref{Decay2beta}).
\end{proof}

\begin{remark}\label{condition2prop3}  The condition $k(k-1)+1.13 c^2\leq \chi_n(c)$ of the previous proposition can be replaced with the following  more explicit condition.
Consider  real numbers  $A, B>1$ with $A^2+B^2 \geq A^2 B^2.$ By using  (\ref{bounds2-chi}), one concludes that  if $n\geq c A$ and $k\leq n/B,$ then  the conditions for \eqref{Decay2beta}  are satisfied.
In particular, one may take $A=B \geq \sqrt{2}.$
\end{remark}

In order to get from \eqref{Decay2beta}, a quantitative decay estimate for the Legendre
expansion coefficients $(\beta_k^n)_k,$ one needs some precise behaviour as well as the decay rate  of the eigenvalues
$\mu_n(c)$ or $\lambda_n(c).$ These issues have been the subjects of many theoretical and numerical studies. To cite but a few
\cite{Bonami-Karoui2, Landau1, LW,   Slepian2, Widom}. In particular in \cite{Bonami-Karoui2},
it has been shown that
\begin{equation}
\lambda_n(c)\leq A(n,c)\left(\frac{ec}{2(2n+1)}\right)^{2n+1},\quad A(n, c)=\delta_1 n^{\delta_2}\left(\frac c{c+1}\right)^{-\delta_3}e^{+\frac {\pi^2}4 \frac{c^2}{n}}.
\end{equation}
As a consequence of the previous inequality, we have the following lemma borrowed  from  \cite{Bonami-Karoui2}.
 \begin{lemma} There exist   constants $a>0$ and $\delta\geq 1$ such that, for $c\geq 1$ and $n>1.35\, c$, we have
 \begin{equation}\label{best-bis}
 \lambda_n(c)\leq \delta e^{-an}.
\end{equation}
\end{lemma}

\section{Quality of the spectral approximation by the PSWFs}

In this section, we first briefly recall the quality of approximation of
band-limited and almost band-limited functions by the classical PSWFs, $\psi_{n}$
that are  concentrated  on $[-b,b],$ for some $b>0.$ These results are mainly due to Landau, Pollack and Slepian, see \cite{Landau, Slepian1}.
For the sake of convenience, we also give some hints to the proofs of these results.
 Then, we show how to extend this study to the case of periodic and non periodic Sobolev
space $H^s([-1,1]), s>0.$

\subsection{Approximation of almost time and band-limited functions}

In this paragraph, $\|\cdot\|_2$ denotes the norm in $L^2(\R)$.
We show that the set $\{\psi_{n}(x),\,\,
n\geq 0\}$ is well adapted for the representation of almost
time-limited and almost band-limited functions, which are defined
as follows.

\begin{definition}
 Let $T=[-a,+a]$ and $\Omega=[-b, +b]$ be two intervals. A
function  $f$, which we assume to be normalized in such a way that
$\|f\|_2=1$, is said to be $\epsilon_T-$concentrated in $T$ and
$\epsilon_{\Omega}-$band concentrated in $\Omega$ if
$$\int_{T^c} |f(t)|^2\, dt \leq \epsilon_T^2,\qquad
\frac 1{2\pi}\int_{\Omega^c} |\widehat f(\omega)|^2\, d\omega \leq
\epsilon_{\Omega}^2.$$
\end{definition}

Up to a re-scaling of the function $f$, we can always assume that
$T=[-1, 1]$ and $\Omega=[-c, +c]$, with $c=ab$. Indeed, for $f$
that is $\epsilon_T-$concentrated in $T=[-a, +a]$ and
$\epsilon_{\Omega}-$band concentrated in $\Omega=[-b, +b]$, the
normalized function $g(t)=\sqrt a f(at)$ is
$\epsilon_T-$concentrated in $[-1, +1]$ and
$\epsilon_{\Omega}-$band concentrated in $[-ab, +ab]$.

Before stating the theorem, let us give some notations. For $f\in L^2(\mathbf R),$
 we consider its expansion $f=\sum_{n\geq
0} a_n \psi_{n,c}$ in $L^2([-1,+1]).$ Due to the
normalization of the functions $\psi_{n,c}$ given by \eqref{eqq1.4}, the following equality holds,
\begin{equation}\label{plancherel-1+1}
   \int_{-1}^{+1}|f(t)|^2 dt=\sum_{n\geq 0}|a_n|^2.
\end{equation}
 We call $S_{N,c}f,$  the $N$-th partial sum, defined by
\begin{equation}\label{plancherel}
   S_{N,c}f(t)=\sum_{n<N}a_n \psi_{n,c}(t).
\end{equation}
 We write more simply $S_Nf$ when there is no ambiguity. In the next lemma, we prove that $S_N f$ tends to $f$ rapidly
 when $f$ belongs to the space of band-limited functions. This statement is both very simple and classical, see for instance \cite{Landau, Slepian1}.

 \begin{lemma}\label{band-limited} Let $f\in B_c$ be an $L^2$ normalized function. Then
 \begin{equation}\label{eq6.0}
 \int_{-1}^{+1}|f- S_Nf|^2 dt\leq   \lambda_N(c).
\end{equation}
\end{lemma}
 \begin{proof} Since the set of functions $\psi_{n,c}$ is also an orthogonal basis of $B_c$, the function $f$ may be written on $\R$ as $f=\sum_{n\geq 0} a_n \psi_{n,c}$, with
  \begin{equation}\label{plancherelR}
   \int_{\R}|f(t)|^2 dt=\sum_{n\geq 0}|\lambda_n(c)|^{-1}|a_n|^2.
\end{equation}
The two expansions coincide on $[-1, +1]$, and, from
\eqref{plancherelR} applied to $f-S_N f$, it follows that
$$\int_{-1}^{+1}|f- S_Nf|^2 dt\leq  \sup_{n\geq N} |\lambda_n(c)| \sum_{n\geq N}|\lambda_n(c)|^{-1}|a_n|^2.$$
We use the fact that the sequence $ |\lambda_n(c)|$ decreases and
\eqref{plancherelR} to conclude.
\end{proof}

Next we define  the time-limiting operator $P_T$ and the
band-limiting operator $\Pi_{\Omega}$ by:
$$ P_T(f) (x) = \chi_T(x) f(x),\qquad
\Pi_{\Omega}(f)(x)=\frac{1}{2\pi}\int_{\Omega} e^{i x
\omega}\widehat{f}(\omega)\, d\omega.$$ The following proposition
provides us with the quality of approximation of almost time- and
band-limited functions by the PSWFs.
\begin{proposition}
If $f$ is an $L^2$ normalized function that is
$\epsilon_T-$concentrated in $T=[-1, +1]$ and
$\epsilon_{\Omega}-$band concentrated in $\Omega=[-c, +c],$ then
for any positive integer $N,$ we have
\begin{equation}\label{eqqq6.0}
\left( \int_{-1}^{+1}|f- S_Nf|^2 dt\right)^{1/2} \leq  \epsilon_{\Omega}+\sqrt{\lambda_N(c)}
 \end{equation}
and, as a consequence,
\begin{equation}\label{eeqq6.0}
\|f - P_T S_N f\|_2 \leq
\epsilon_T+\epsilon_{\Omega}+\sqrt{\lambda_N(c)}.
\end{equation}
More generally, if $f$ is an $L^2$ normalized function that is
$\epsilon_T-$concentrated in $T=[-a, +a]$ and
$\epsilon_{\Omega}-$band concentrated in $\Omega=[-b, +b]$ then,
for $c=ab$ and for any positive integer $N,$ we have
\begin{equation}\label{eeqq6.01}
\|f - P_T S_{N,c,a}f\|_2 \leq \epsilon_T+\epsilon_{\Omega}+\sqrt
{\lambda_N(c)}
\end{equation}
where $S_{N,c,a}$ gives the $N$-th partial sum for the orthonormal
basis $\frac 1{\sqrt a}\psi_{n,c}(t/a)$ on $[-a, +a]$.
\end{proposition}

\begin{proof}[{\bf Proof:}] We first  prove \eqref{eqqq6.0} by writing $f$ as the sum of $\Pi_\Omega f$ and $g$.
Remark first that $\int_{-1}^{+1}|g-S_N g|^2 dt\leq \|g\|_2 \leq \epsilon_\Omega$. We then use Lemma \ref{band-limited}
for the band-limited function $\Pi_\Omega f$ to conclude. The rest of the proof follows at once.
\end{proof}

\begin{remark} Let $f$ be a normalized $L^2$ function that vanishes outside $I$ and  we assume that
 $f\in H^s(\R)$.
Then $f$ gives an example of $0$-concentrated in $I$ and $\epsilon_c$-band concentrated in $[-c, +c]$,
with $\epsilon_c\leq M_f/c^s$ and ${\displaystyle M_f^2=\frac 1 {2\pi}\int |\widehat f(\xi)|^2|\xi|^{2s} d\xi.}$
\end{remark}

\subsection{Approximation by the PSWFs in Sobolev spaces}

In this paragraph, we study the quality of approximation by the
PSWFs in the Sobolev space $H^s([-1,1]).$  We
 provide an $L^2([-1,1])$-error bound of the approximation
of a function  $f\in H^s([-1,1])$ by the $N-$th partial sum of its
expansion in the   PSWFs basis. To simplify notation we will write $I=[-1,1]$.

We should mention that
different spectral approximation results by the PSWFs in $H^s(I)$ have
been already given in \cite{Boyd1, Chen, Wang}.  It is important to mention that the  error bounds of the spectral
approximations given by the previous references  do not indicate
how to choose a    ``good" value of the bandwidth $c$
to approximate a given $f\in H^s(I).$ By a simultaneous use of the
properties of the PSWFs as eigenfunctions of the differential
operator $\mathcal L_c$ and the integral operator $\mathcal F_c,$ we give a first
answer to this question. This is the subject of the following
theorem.

\begin{theorem}\label{approx1}
Let  $c > 0$ be a positive real number. Assume that $f\in
H^s(I)$, for some positive real number $s>0$. Then for any
integer $N\geq 1,$ we have
\begin{equation}\label{eq222.1}
\| f-S_N f\|_{L^2(I)}\leq K(1+c^2)^{-s/2} \|
f\|_{H^s(I)}+ K\sqrt{ \lambda_N(c)} \|f\|_{L^2(I)}.
\end{equation}
Here, the constant $K$  depends only on $s$ and on the extension operator from
$H^s(I)$ to $H^s(\mathbb R).$ Moreover it can be
taken equal to $1$ when $f$ belongs to the space $H^s_0(I)$.
\end{theorem}

\begin {proof} To prove  (\ref{eq222.1}), we first use the fact that for
any real number $s\geq 0,$ there exists a linear and continuous
extension operator $E: H^s(I)\rightarrow H^s(\mathbb R).$
Moreover, if $f\in H^s(I)$ and $F= E(f) \in H^s(\mathbb R),$
then there exists a constant $K>0$ such that
\begin{equation}\label{extension}
\| F \|_{L^2(\mathbb R)} \leq K \| f\|_{L^2(I)},\qquad \|
F\|_{H^s(\mathbb R)} \leq K \|f\|_{H^s(I)}.
\end{equation} We recall that the Sobolev norm of a function F on $\R$ is given by
$$\|F\|_{H^s(\R)}^2 =\frac{1}{2\pi}\int_{\R} (1+|\xi|^2)^s|\widehat f (\xi)|^2\, d\xi.
$$
In particular, for a $c-$bandlimited function $F$ , one has
$$\|F\|_{L^2(\R)}^2\leq (1+c^2)^{-s}\|F\|_{H^s(\R)}^2.$$
Next,
if ${\mathcal F}$ denotes the Fourier transform operator and if
$${\cal G} = {\cal F}^{-1}(\widehat F \cdot 1_{[-c,c]}),\quad {\cal
H}={\cal F}^{-1}(\widehat F \cdot (1-1_{[-c,c]})),$$ then ${\cal
G}$ is $c-$bandlimited and $F={\cal G} + {\cal H}.$ Moreover,
since $\| \widehat {\mathcal G}\|_{L^2(\mathbb R)}\leq \|\widehat
F\|_{L^2(\mathbb R)}$ and $\|  {\mathcal H}\|_{L^2(\mathbb R)}\leq
c^{-s} \|F\|_{H^s(\mathbb R)},$ then by using (\ref{extension}),
one gets
\begin{equation}\label{eq222.3}
\| {\mathcal G}\|_{L^2(\mathbb R)}\leq K \|f\|_{L^2(I)},\quad \|
{\mathcal H}\|_{L^2(I)}\leq K (1+c^2)^{-s/2} \| f\|_{H^s(I)}.
\end{equation}
Finally, by using the previous inequalities and the fact that
${\mathcal G}$ is $c-$bandlimited, one concludes that
\begin{eqnarray*}
\| f- S_N f\|_{L^2(I)}&\leq & \|\mathcal G - S_N \mathcal G\|_{L^2(I)}
+\|\mathcal H - S_N \mathcal H\|_{L^2(I)}\\
&\leq & \sqrt{ \lambda_N(c)} \|{\mathcal
G}\|_{L^2(\mathbb R)} +\|\mathcal H\|_{L^2(I)}\\
&\leq &  \sqrt{ \lambda_N(c)} K \|f\|_{L^2(I)} +
 K (1+c^2)^{-s}\| f\|_{H^s(I)}.
\end{eqnarray*}

This concludes the proof for general $f$. When $f$ is in the subspace $H^s_0(I)$, one can take as
extension operator the extension by $0$ outside $I$, so that the constant $K$ can be replaced by $1$.
\end{proof}



In \cite{Wang}, the author has used a different approach for the
study of the spectral approximation by the PSWFs. More precisely,
by considering the weighted Sobolev space $\widetilde H^r(I),$
associated with the differential operator ${\mathcal L}_c$
defined by
$$\widetilde H^r(I)= \left\{ f\in L^2(I),\, \|f\|^2_{\widetilde H^r(I)}=\|{\mathcal L}_c^{r/2} f\|^2=\sum_{k\geq 0}
(\chi_k)^r | f_k|^2 <+\infty\right\},$$ where $f=\sum f_k \psi_k$ is the expansion in the basis of PSWFs. Then for any $f\in \widetilde H^r(I),$ with $r\geq 0,$ we have
$$\|f-S_N f\|_{L^2(I)} \leq (\chi_N(c))^{-r/2} \|f\|_{\widetilde H^r(I)}\leq N^{-r} \|f\|_{\widetilde H^r(I)}.$$
For more details, the reader is referred to \cite{Wang}.

\begin{remark}
Compared the result of Theorem \ref{approx1}  with Wang's result, this latter  has the advantage to give an error term for all
values of $c$, while the first term in \eqref{eq222.1} is only
small for $c$ large enough. On the other hand, Wang compares his
specific Sobolev space with the classical one and finds that
$$\|f\|_{\widetilde{H^s}(I)}\leq C(1+c^2)^{s/2}\|f\|_{H^s(I)}.$$
For large values of $N$ we clearly have  ${\displaystyle \frac{(1+c^2)}{\chi_N}\ll
(1+c^2)^{-1},}$ but it goes the other way around when $\chi_N$ and
$1+c^2$ are comparable. So it may be useful to have both kinds of
estimates in mind for numerical purpose and for the choice of the
value of  $c.$
\end{remark}

\begin{remark}
 The error bound given by Theorem \ref{approx1} has the advantage to be explicitly given in terms of
 $c$ and $\lambda_n(c).$ Nonetheless, it has a drawback that it does not imply a rate of convergence, nor even the  convergence
 of $S_N(f)$ to $f$ in the usual $L^2(I)-$norm. To overcome this problem, we devote the remaining of this section
 to a more elaborated convergence analysis in the $2$-periodic Sobolev space $H^s_{per},$ then we extend this analysis
 to the usual $H^s(I)-$space.
\end{remark}

Next, we  consider  the subspace $H^s_{per}$ of functions in $H^s(I)$ that extend into $2-$periodic functions of the same regularity.
For such functions, one can also use the norm
$${\displaystyle \|f\|_{H^s_{per}}=\sum_{k\in \mathbb Z}
(1+(k\pi)^2)^s | b_k(f)  )|^2.}$$
 Here, $$ b_k(f)=\frac 1{\sqrt 2} \int_{-1}^{+1}f(x) e^{-i\pi k x} dx=\frac 1{\sqrt 2}\widehat f (k\pi) $$
is the coefficient of the Fourier series expansion of $f.$ We then have the following theorem.

 \begin{theorem}
Let $c\geq 1,$   then there exist constants $M\geq \sqrt{2}$ and $ M',\, a>0$ such that, when  $N\geq \max (c M,3)$ and  $f\in H^s_{per}, s>0$,  we have the inequality
\begin{equation}\label{error1}
\|f- S_N (f)\|_{L^2(I)}\leq M' (1+(\pi N)^2)^{-s/2}  \| f\|_{H^s_{per}}+ M' e^{-aN}  \| f\|_{L^2}.
\end{equation}
\end{theorem}
\begin{proof}  We start with reductions of the problem, which are analogous to the ones that we have detailed above. It is sufficient to prove this separately with the constant $M'/2$ for periodic functions $g$ and $h=f-g$, where $g$ is the projection of $f$ onto the subspace of $H^s_{per}$ whose Fourier coefficients $b_k(f)$ are zero for $|k|>N/M.$ Moreover, we have directly the inequality without a second term, since the $L^2$ norm of $h$ may be  bounded by the first term multiplied by some constant. So, let us prove the inequality for $g$. This time we will prove that the inequality holds without the first  term, that is,
$$\|g- S_N (g)\|_{L^2(I)} \leq   \frac{M'}{2} e^{-aN}  \| g\|_{L^2(I)}.$$
The next reduction consists of restricting to exponentials $e^{ik\pi x}$, with $|k|\leq N/M$. Indeed, assume that we prove the previous  inequality  for all of them, with a uniform bound by $M''e^{-a'N}$. Then, by linearity we  will have
$$\|g- S_N (g)\|_{L^2(I)} \leq M''e^{-a'N}\sum | b_k(g)|\leq  M''e^{-a'N} \sqrt{2[N/M]+1}\;e^{-aN} \| g\|_{L^2(I)}.$$
This in turn gives up to a constant,  the required form by choosing $a<a'$.

So we  content ourselves to consider $f(x)=e^{ik\pi x}$, with $|k|\leq N/M$. Finally, since $$\|f-S_N f\|_{L^2(I)}^2=\sum |\langle f, \psi_n\rangle|^2,$$
then  it is sufficient to have such an estimate for each $n>N$, and conclude by taking the sum $\sum_{n> N} e^{-an}$. So the proof is a consequence of the following lemma.
\end{proof}
The following well known identity  will be needed in the sequel.
\begin{equation}
\label{FourierPn}
\int_{-1}^1 e^{i \lambda x}
\overline{P_{n}}(x)\, dx= i^n
\sqrt{\frac{2\pi}{\lambda}}\sqrt{n+1/2} J_{n+1/2}(\lambda),
\end{equation}
where $J_{n+1/2}(\cdot)$ is the Bessel function of the first type and order $n+1/2.$

\begin{lemma}
Let $c\geq 1,$ then there exist constants $M \geq  \sqrt{2}$ and $  M', \, a>0$ such that, when  $n\geq \max\left( c M, 3\right)$ and  $f(x)=e^{ik\pi x}$ with $|k|\leq n/M,$ we have
\begin{equation}
|\langle f, \psi_n\rangle|\leq M'e^{-an}.
\end{equation}
\end{lemma}

\begin{proof} This scalar product can be written  by using \eqref{FourierPn}
\begin{eqnarray*}
< e^{ik\pi x},\psi_n >&=& \int_{-1}^1  e^{ik\pi x}\,\psi_n(x)\, dx = \sum_{m\geq 0} \beta_m^n < e^{ik\pi x},\overline P_m > =\sum_{m\geq 0} \beta_m^n \sqrt{\frac{2}{k}}\sqrt{m+1/2} J_{m+1/2} (k\pi)\\
&=&\sum_{m= 0}^{[n/M]} \beta_m^n \sqrt{\frac{2}{k}}\sqrt{m+1/2} J_{m+1/2} (k\pi)+\sum_{m\geq [n/M]+1} \beta_m^n \sqrt{\frac{2}{k}}\sqrt{m+1/2} J_{m+1/2} (k\pi)\\
&=&I_1^n+ I_2^n.
\end{eqnarray*}
To bound $I_1^n,$ we first remark   that  the Fourier transform of $\overline P_n\chi_{[-1, 1]}$ is bounded by $1$ and then
we use remark \ref{condition2prop3} to check that \eqref{Decay2beta} is satisfied whenever $n\geq c M$ with $M\geq 1.40.$
Hence, we have
\begin{eqnarray*}
|I_1^n|&\leq& \sum_{m= 0}^{[n/M]} |\beta_m^n| \leq  \sqrt{\frac{5}{4\pi}}|\mu_n(c)| \sum_{m= 0}^{[n/M]}
\left(\frac{2\sqrt{\chi_n}}{c}\right)^m\\
&\leq & K  \left(\frac{2\sqrt{\chi_n}}{c}\right)^{[n/M]+1} |\mu_n(c)|.
\end{eqnarray*}
Moreover, taking into account the decay of the $\mu_n(c)$ given by \eqref{best-bis} and using the upper bound
of $\chi_n,$ given by \eqref{bounds2-chi}, we conclude that
\begin{equation}
\label{boundI1}
|I_1^n| \leq K' \left(\frac{\pi(n+1)}{c}\right)^{\frac nM +1}e^{-\delta n}\leq K'' e^{-an}
\end{equation}
for some   sufficiently small positive real number $a$, as soon as $M\geq \sqrt{2}$.
To bound $I_2^n,$ it suffices to use the fact that $|\beta_k^n|\leq 1$ and the bound
 of the Bessel function given by \cite{Andrews},
 \begin{equation}\label{Eq2.2}
 |J_{\alpha}(x)|\leq \frac{|x|^{\alpha}}{2^{\alpha}
 \Gamma(\alpha+1)},\quad \forall\, \alpha > -1/2,\quad\forall\,
 x\in \mathbb{R}.
 \end{equation}
 One concludes that
 \begin{eqnarray*}
 |I_2^n|&\leq &\sum_{m\geq n/M}  \sqrt{2/k} \sqrt{m+1/2} |J_{m+1/2} (k\pi)|\leq \sum_{m\geq [n/M]+1} \sqrt{2/k}\sqrt{m+1/2}
 \frac{ (k\pi)^{m+1/2}}{2^{m+1/2} \Gamma(m+3/2)} \\
 &\leq & \sqrt{\pi}\sum_{m\geq [n/M]+1}\frac{ (k\pi)^m}{2^m\sqrt{m+1/2}\Gamma(m+1/2)}.
 \end{eqnarray*}
Moreover, since ${\displaystyle \Gamma(m+1/2)\geq m!/\sqrt{m+1}}$ and ${\displaystyle m! \geq (m/e)^m \sqrt{2\pi m}},$ each term is bounded by an exponential $e^{-an}$ and we find the required estimate for $|I_2^n|$.
\end{proof}

In \cite{Chen}, the authors have given a different $L^2(I)-$convergence rate of $S_N(f)$ to $f$
in terms of the decay of the expansion coefficients $a_k(f)=\int_{-1}^1 f(x)\psi_{k}(x)\, dx.$ More precisely,
it has been shown in \cite{Chen} that
$$|a_N(f)|\leq C \left( N^{-2/3 s}\| f\|_{H^s(I)}+\left(\sqrt{\frac{c^2}{\chi_N(c)}}\right)^{\delta N}\| f\|_{L^2(I)}\right),$$
where $C, \delta$ are independent of $f, N$ and $c.$

\begin{remark}
The previous theorem gives the rate of convergence of the truncated PSWFs series expansion of a function $f$ from  $H^s_{per}.$ This rate of convergence
will be generalized in the sequel to the usual $H^s(I)-$space. Note that this rate of convergence drastically improves the one given by \cite{Chen}.
Moreover, unlike the error bound given in \cite{Wang}, the decay of the error bound given by the previous  theorem is still valid even when $N$ is comparable to $c.$
 Nonetheless, in practice,  Theorem \ref{approx1} is useful in the sense that it provides us
with a criteria for the choice of the bandwidth $c>0,$ that depends on the magnitude of the Sobolev exponent $s>0.$ The smaller $s,$ the larger $c$ should be and vice versa.
\end{remark}

\begin{remark}
We also have a bound of the error for ordinary polynomials.
Indeed, if we consider the monomial $f(x):=x^j$, then
$$a_n(f)=\int_{-1}^1 y^j \psi_{n,c}(y)\, dy= (-i)^j c^{-j} \mu_n(c)\psi_{n,c}^{(j)}(0),\quad\mbox{with } i^2=-1.$$
For fixed $j,$ we can then use Corollary 1  to conclude that if $c\geq 1$ and $c^2/\chi_N
<1,$ then we have
\begin{equation}\label{majoration1}
\| f- S_N f\|_2^2 \leq M \sum_{k\geq N} \left(\frac{\chi_k(c)}{c^2}\right)^j|\mu_k(c)|^2\leq M' c^{-2N} \sum_{k\geq N} k^{2j} e^{-ak},
\end{equation}
which also leads to an exponential decay.
\end{remark}

As a corollary of the previous theorem and remark, we obtain the following corollary that extends the result of the
previous theorem to the case of the usual Sobolev space $H^s([-1,1]).$

\begin{corollary}
Let $c\geq 1,$  and let  $s>0.$ There exist constants $M\geq \sqrt{2}$ and $M', M'_s>0$ such that, when
$f\in H^s(I)$ and $N\geq \max\left( c M, 3\right)$,   we have the inequality
\begin{equation}\label{error2}
\|f- S_N (f)\|_{L^2(I)}\leq M'_s (1+ N^2)^{-s/2}  \| f\|_{H^s([-1,1])}+ M' e^{-aN}  \| f\|_{L^2([-1,1])}.
\end{equation}
\end{corollary}

\begin{proof} We first assume that
$[s]= m,$ and $s\not\in \frac{1}{2}+\mathbf N,$ then there exists
a polynomial $P$, of degree  at most $m$,  such that $f+ P\in H^s_{per}.$
Consequently, by using the previous theorem and the inequality (\ref{majoration1}), one concludes
for (\ref{error2}). More generally, a function $f\in H^s(I)$ can be considered as  the restriction to $I$ of a function
of $H^s(\mathbb R)$ which may be taken to have support in $[-2,2].$ So it is also the restriction to $I$ of a periodic function of period 4.
Since Lemma 4 is also valid for the exponentials ${\displaystyle e^{i\frac{k\pi}{2} x},\, k\in \mathbb Z,}$ then we conclude as before.
\end{proof}

\section{Numerical results}

In this section, we illustrate the results of the previous  sections
by various numerical examples. For this purpose,
we first describe a numerical method for the
computation of the  PSWFs series
expansion coefficients of a function from the Sobolev space
$H^s(I).$ Note that if $f\in H^s_{per},\, s>0,$ then its different
PSWFs series expansion coefficients $(a_n(f))_n$ can be easily approximated as follows.
For a positive
integer $K,$  an approximation $a_n^K(f)$ to $a_n(f)$ is given by
the following formula
\begin{equation}\label{eq3.10}
a_n^K(f)= \frac{\mu_n(c)}{\sqrt{2}} \sum_{k=-K}^K b_k(f)
\psi_{n,c}\left(\frac{k\pi}{c}\right)= a_n(f) + \epsilon_K,
\end{equation}
where the $b_k(f)$ are the  Fourier coefficients of $f$ and
where ${\displaystyle \epsilon_K=\frac{1}{\sqrt{2}}\sum_{|k|\geq
K+1}\mu_n(c)b_k(f)
\psi_{n,c}\left(\frac{k\pi}{c}\right).}$ Moreover,  from the well
known asymptotic behavior of the $\psi_{n,c}(x),$ for large values
of $x,$ see for example \cite{Karoui1}, one can easily check  that
${\displaystyle \epsilon_K =
o\left(\frac{1}{((K+1)\pi)^{1+s}}\right).}$ This computational
method of the $a_n(f)$  has the advantage to work for small as
well as large values of the smoothness coefficient $s>0.$\\

Also, note that if $f\in H^s([-1,1]),$ where $s > 1/2+ 2m, m\geq 1,$ is an
integer,  then $f\in C^{2m}([-1,1]).$ Moreover since
$\psi_{n,c}\in C^{\infty}(\mathbb R),$ then the classical Gaussian
quadrature method, see for example \cite{Andrews} gives us the
following approximate value $\widetilde a_n(f)$ of the $(n+1)-$th
expansion coefficient $a_n(f)=<f,\psi_{n,c}>,$
\begin{equation}\label{eq3.8}
\widetilde a_n(f)=\sum_{l=1}^m \omega_l f(x_l)\psi_{n,c}(x_l)=
a_n(f)+\epsilon_n,
\end{equation}
with  ${\displaystyle |\epsilon_n| \leq \sup_{\eta\in
[-1,1]}\frac{1}{b_m^2}\frac{(f \cdot
\psi_{n,c})^{(2m)}(\eta)}{(2m)!}.}$ Here, $b_m$ is the highest
coefficient of $\overline{P_m},$ and the different weights
$\omega_l$ and nodes $x_l,$ are easily computed by the special
method given in \cite{Andrews}.

The following examples illustrate the quality of approximation in $H^s(I)$ by the PSWFs.\\

\noindent {\bf Example 1:} In this example, we show that the PSWFs
outperforms the Legendre polynomials in the approximation of a
class of functions from the Sobolev space $H^s([-1,1]),$ having
significant  large coefficients at some high frequency components.
 To fix the idea, let $\lambda>0,$ be a relatively large positive
real number and let $f_{\lambda}(x)=e^{i\lambda x}, x\in [-1,1].$
The Legendre series expansion coefficients of $f_{\lambda}$ are
given by
$$\alpha_{n}(0)= \int_{-1}^1 e^{i \lambda x}
\overline{P_{n}}(x)\, dx= i^n
\sqrt{\frac{2\pi}{\lambda}}\sqrt{n+1/2} J_{n+1/2}(\lambda).$$ In this case, we
have
\begin{equation}\label{eq5.0}
\|f_{\lambda}-\sum_{n=0}^N
\alpha_{n}(0)\overline{P_n}\|_2^2=\frac{2\pi}{\lambda}\sum_{n\geq
N+1} (n+1/2) (J_{n+1/2}(\lambda))^2. \end{equation}
 If $c>0$ is a positive
real number, then  the corresponding  PSWFs series expansion
coefficients of $f_{\lambda}$ are simply given as follows,
$$\alpha_{n}(c)=\int_{-1}^1 e^{i\lambda x}\psi_{n,c}(x)\, dx=
\mu_n(c) \psi_{n,c}(\lambda/c).$$
Note that the analytic extension of $\psi_{n,c}$ outside the interval $[-1,1]$
has been given in \cite{Slepian1} as follows
\begin{equation}\label{eq2.2.10}
\psi_{n}(x)= \frac{\sqrt{2\pi}}{|\mu_n(c)|}{\sum_{k\geq 0}} (-1)^k
\beta_k^n \sqrt{k+1/2}\frac{J_{k+1/2}(c
x)}{\sqrt{cx}},
\end{equation}
with
\begin{equation}\label{eq2.2.11}
 \mu_n(c)= i^n \sqrt{\frac{2\pi }{c}} \left[\frac{{\sum_{k\geq 0}} (-1)^k \sqrt{k+1/2}
\,\,\beta_k^n \,\, J_{k+1/2}(c)}{ {\sum_{k\geq 0}}
\beta_k^n\sqrt{k+1/2}}\right],
\end{equation}
is the exact value of the $n-$th eigenvalue of the finite Fourier
transform operator $\mathcal F_c.$

On the other hand, the
$L^2(I)-$approximation error by the PSWFs is given by
\begin{equation}\label{eeq5.1}
E_N(c)=\|f- \sum_{n=0}^N \alpha_n(c) \psi_{n,c}\|_2^2=\sum_{n\geq
N+1} |\mu_n(c)|^2
\left(\psi_{n,c}\left(\frac{\lambda}{c}\right)\right)^2.
\end{equation}
In the special case where $c=\lambda,$ the previous error bound
becomes ${\displaystyle E_N(\lambda)=\sum_{n\geq N+1} |\mu_n(c)|^2
\left(\psi_{n,c}(1)\right)^2.}$ Since from \cite{Bonami-Karoui1}, we have $|\psi_{n,c}(1)|\leq 2 \chi_n^{1/4},$ then by using
\eqref{bounds2-chi}, one gets $|\psi_{n,c}(1)|\leq \sqrt{2\pi (n+1)}.$
Moreover, since  the super-exponential
decay of the sequence  $(|\mu_n(c)|^2)_{n\geq 0}$ starts  around
$N_c=[ec/4],$  then from (\ref{eq5.0}) and (\ref{eeq5.1}), one
concludes that  the PSWFs are better adapted for the approximation
of the $f_{\lambda}$ by its $N-$th order truncated PSWFs series
expansion with $c=\lambda$ and $N=[\lambda].$ More generally, if
$0\leq c < \lambda,$ then $\frac{\lambda}{c} >1$ and the well known blow-up of
the $\psi_{n,c}\left(\frac{\lambda}{c}\right)$ with $\frac{\lambda}{c} >1,$ implies that
$\alpha_n(c) = \mu_n(c) \psi_{n,c}(\lambda/c)$ has a lower decay
than $\alpha_n(\lambda) = \mu_n(\lambda) \psi_{n,c}(1).$ Moreover,
if $c> \lambda,$ then the decay of the $|\mu_n(c)|^2$ and
consequently, the fast decay of the $\alpha_n(c)$ is possible only
if $n$ lies beyond a neighbourhood of $\frac{e c}{4} >
\frac{e\lambda}{4}.$ This means that $c=\lambda$ is the
appropriate value of the bandwidth to be used to approximate the
function $f_{\lambda}(x)= e^{i\lambda x}$ by its first $N-$th
truncated PSWFs series expansion, with $N=[\lambda].$ This
explains the numerical results given in \cite{Wang} concerning the
approximation of the test function $u(x)=\sin(20\pi x),$ where the
author has checked numerically that $c=20\pi$ is the appropriate
value of the bandwidth for approximating $u(x)$ by the PSWFs
$\psi_{n,c}$ with a given high precision and minimal number of the
truncation order $N.$ As another  example, we consider the value
of  $\lambda=50,$ then we find that
$$\| f_{\lambda}-\sum_{n=0}^{50}\alpha_n(0) \overline{P_n}\|_2\approx
3.087858E-01,\quad \|f_{\lambda}-\sum_{n=0}^{50}\alpha_n(50)
\psi_{n,50}\|_2\approx 1.356604 E-08.$$

\noindent {\bf Example 2:} In this example, we consider   the
Weierstrass function
\begin{equation}\label{eq5.1}
W_s(x)= \sum_{k\geq 0} \frac{\cos(2^{k}x)}{2^{ks}},\quad -1\leq
x\leq 1.
\end{equation}
The choice of this function allows us to see how the approximation process works on functions that are either nowhere
smooth or having small Sobolev smoothness exponent.
Note that $W_s \in H^{s-\epsilon}([-1,1]),\,\forall \epsilon <
s,\, s
>0.$ We have considered the value of $c=100,$ and computed
$W_{s,N},$ the $N-$th terms truncated PSWFs series expansion of
$W_{s}$ with different values of ${\displaystyle \frac{3}{4}\leq s
\leq 2}$ and different values of $20\leq N \leq 100.$  Also, for
each pair $(s,N),$  we have computed the corresponding approximate
$L^2-$ error bound ${\displaystyle
E_N(s)=\left[\frac{1}{50}\sum_{k=-50}^{50}
(W_{s,N}(k/50)-W_{s}(k/50))^2\right]^{1/2}}.$  Table 1 lists the
obtained values of $E_N(s).$ Note that the numerical results given
by Table 1, follow what has been predicted by the theoretical
results of the previous section. In fact, the $L^2-$errors
$\|W_s-\Pi_N W_s\|_2$ is of order $O(N^{-s}),$ whenever
${\displaystyle N \geq N_c \sim \left[\frac{ 2c}{\pi}\right]+4.}$  The graphs of $W_{3/4}(x)$ and $W_{3/4,
N}(x),\, N=90$ are given by Figure 1.
\begin{center}
\begin{tiny}
{\tiny
\begin{table}[]
\caption{Values of $E_N(s)$ for various values of $N$ and $s.$}
\vskip 0.2cm
\begin{tabular}{ccccccc} \hline
    & $s=0.75$ &              $s= 1$    &     $s= 1.25$ &   $ s=1.5$ &   $ s=1.75$ & $s= 2.0$ \\ \hline
$N$ &  $E_N(s)$ & $E_N(s)$ &$E_N(s)$ &$E_N(s)$ &$E_N(s)$ &$E_N(s)$
\\ \hline
20& 4.57329E-01& 4.66173E-01 & 4.85990E-01  &    5.05973E-01  &    5.23232E-01  &  5.37227E-01 \\
30& 3.15869E-01& 3.11677E-01 & 3.28241E-01  &    3.48562E-01  &    3.67260E-01  &  3.82963E-01 \\
40& 1.06843E-01& 1.52009E-01 & 1.91237E-01  &    2.20969E-01  &    2.43432E-01  &  2.60523E-01 \\
50& 4.09844E-02& 6.88472E-02 & 1.01827E-01  &    1.26518E-01  &    1.44809E-01  &  1.58520E-01 \\
60& 3.30178E-02& 2.09084E-02 & 3.25551E-02  &    4.28999E-02  &    5.06959E-02  &  5.65531E-02  \\
70& 3.15097E-02& 8.82446E-03 & 2.51157E-03  &    7.35725E-04  &    2.33066E-04  &  1.04137E-04 \\
80& 3.01566E-02& 8.55598E-03 & 2.40312E-03  &    6.87458E-04  &    1.98993E-04  &  5.80481E-05  \\
90& 2.67972E-02& 7.64167E-03 & 2.14661E-03  &    6.15062E-04  &    1.78461E-04  &  5.22848E-05 \\
100&2.39141E-02& 6.72825E-03 & 1.82818E-03  &    5.10057E-04  &    1.45036E-04  &  4.19238E-05  \\

 \hline
\end{tabular}\end{table}
}
\end{tiny}
\end{center}

\begin{figure}[h]\hspace*{0.5cm}
{\includegraphics[width=14.5cm,height=5.5cm]{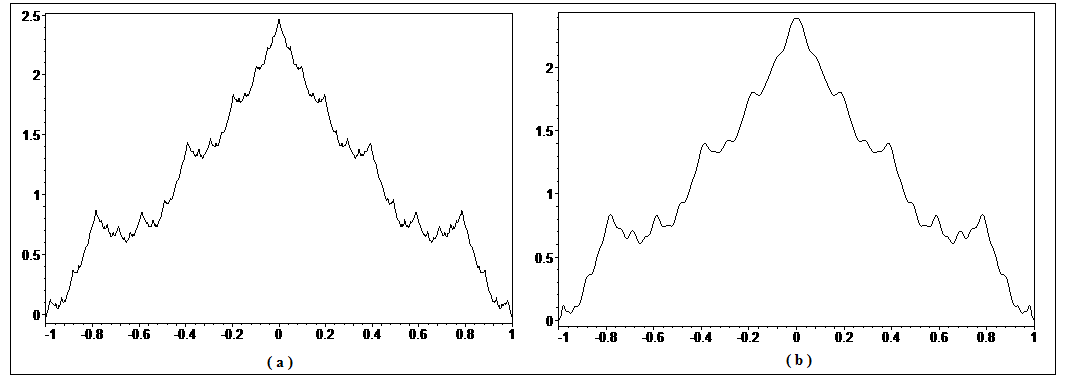}} \vskip
-0.5cm\hspace*{2cm}\caption{(a) graph of $W_{3/4}(x),\quad$ (b)
graph of $W_{3/4,N}(x), N=90.$ }
\end{figure}

\noindent {\bf Example 3:}  In this example, we let $s>0$ be any
positive real number and we consider the random  function
$B_s(x)$ is given as follows.
\begin{equation*}
B_s(x)= \sum_{k\geq 1} \frac{X_k}{k^s} \cos(k\pi x),\quad -1\leq
x\leq 1.
\end{equation*}
Here, $X_k$ is a sequence of independent standard Gaussian  random variables.  The random process $B_s$ behaves like a fractional Brownian motion, with Hurst parameter $H=s-1/2.$ It is almost surely in $H^{s'}$ for $s'<s-1/2.$
For the special case $s=1,$  we consider
the band-width $c=100,$ a truncation order $N=80$  and compute
$B_{1,N}$ the  approximation of $B_1$ by its $N-$th terms
truncated PSWFs series expansion. The graphs of $B_1$ and
$B_{1,N}$ are given by Figure 2.

\begin{figure}[h]\hspace*{0.5cm}
{\includegraphics[width=14.5cm,height=6.2cm]{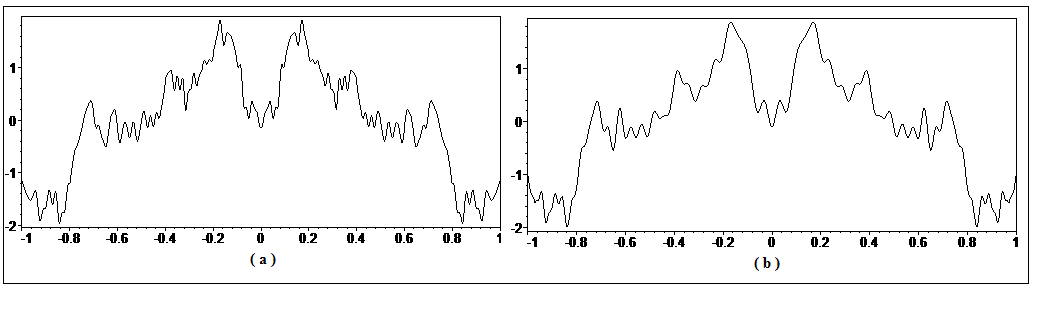}} \vskip
-0.5cm\hspace*{2cm}\caption{(a) graph of $B_1(x),\quad$ (b) graph
of $B_{1,N}(x), N=80.$ }
\end{figure}


\begin{remark}
From the quality of approximation in the Sobolev spaces
$H^s([-1,1])$ given in this paper and in \cite{Boyd1, Chen, Wang},
one concludes that for any value of the bandwidth $c\geq 0,$  the
approximation error $\|f-S_N f\|_2$ has the asymptotic order
$O(N^{-s}).$ Nonetheless, for a given $f\in H^s([-1,1]), s>0$
which we may assume to have a unit $L^2-$norm and for a given
error tolerance $\epsilon,$ the appropriate value of the bandwidth
$c\geq 0,$ corresponding to the minimum truncation order $N,$
ensuring that $\|f-S_N f\|_2\leq \epsilon,$ depends on whether or
not, $f$ has some significant Fourier expansion coefficients,
corresponding to large frequency components. In other words, the
faster decay to zero  of the Fourier coefficients of $f,$ the
smaller the value of the bandwidth should be and vice versa.
\end{remark}

\noindent
{\bf Acknowledgement:} The authors thank very much the anonymous referee for the valuable comments and suggestions
that helped them to improve the revised version of this work. Special thanks of the second author  go to Laboratory MAPMO
of the University of Orl\'eans
where part of this work has been done while he  was a visitor there.

\end{document}